\documentclass[11pt,a4paper]{amsart}

\usepackage[margin=1in]{geometry}
\usepackage[utf8]{inputenc}
\usepackage[T1]{fontenc}
\usepackage{amsmath,amsthm,amssymb,amsfonts}
\usepackage{mathtools,mathrsfs}
\usepackage{enumitem}
\usepackage{booktabs}
\usepackage[colorlinks=true,linkcolor=blue,citecolor=red,urlcolor=magenta]{hyperref}
\usepackage[capitalise,nameinlink,noabbrev]{cleveref}
\usepackage{graphicx}
\graphicspath{{figs/}}

\newcommand{\n}{\nabla}
\renewcommand{\S}{\Sigma}

\usepackage{aliascnt} 
\theoremstyle{plain} 
\newtheorem{theorem}{Theorem}[section] 
\newaliascnt{proposition}{theorem} \newtheorem{proposition}[proposition]{Proposition} \aliascntresetthe{proposition} 
\newaliascnt{lemma}{theorem} 
\newtheorem{lemma}[lemma]{Lemma} \aliascntresetthe{lemma} 
\newaliascnt{corollary}{theorem} \newtheorem{corollary}[corollary]{Corollary} \aliascntresetthe{corollary} \theoremstyle{definition} 
\newaliascnt{definition}{theorem}  \aliascntresetthe{definition} \newaliascnt{example}{theorem}  \aliascntresetthe{example}
\theoremstyle{remark} 
\newaliascnt{remark}{theorem} 
\newtheorem{remark}[remark]{Remark} \aliascntresetthe{remark} 

\crefname{theorem}{teorema}{teoremas} \Crefname{theorem}{Teorema}{Teoremas} \crefname{proposition}{proposição}{proposições} \Crefname{proposition}{Proposição}{Proposições} \crefname{lemma}{lema}{lemas} 
\Crefname{lemma}{Lema}{Lemas} 
\crefname{corollary}{corolário}{corolários} \Crefname{corollary}{Corolário}{Corolários} \crefname{definition}{definição}{definições} \Crefname{definition}{Definição}{Definições} \crefname{remark}{observação}{observações} \Crefname{remark}{Observação}{Observações} \crefname{example}{exemplo}{exemplos} \Crefname{example}{Exemplo}{Exemplos} \crefname{section}{seção}{seções} \Crefname{section}{Seção}{Seções} \crefname{subsection}{subseção}{subseções} \Crefname{subsection}{Subseção}{Subseções} \crefname{equation}{equação}{equações} \Crefname{equation}{Equação}{Equações} \crefname{table}{tabela}{tabelas} 
\Crefname{table}{Tabela}{Tabelas} \crefname{figure}{figura}{figuras} \Crefname{figure}{Figura}{Figuras}

\newcommand{\R}{\mathbb{R}}
\newcommand{\Hyp}{\mathbb{H}}
\newcommand{\C}{\mathsf{C}}
\newcommand{\F}{\mathsf{F}}
\newcommand{\J}{\mathcal{J}}
\DeclareMathOperator{\sech}{sech}

\DeclareMathOperator{\csch}{csch}

\title[Horocyclic CMC  hypersurfaces in $\Hyp^2\times\Hyp^2$]%
{Constant mean curvature hypersurfaces in $\Hyp^2\times\Hyp^2$ with double horocyclic symmetry}

\author{Julio Cesar Mohnsam}
\address{Department of Mathematics, UFRGS, Porto Alegre, Brazil}
\email{julio.mohnsam@ufrgs.br}
\date{\today}

\keywords{Product spaces, CMC hypersurfaces, horocyclic symmetry, first integrals, homogeneous geometries.}
\subjclass[2020]{Primary 53C42; Secondary 53C40, 53A10, 34A34}

\begin{document}

\begin{abstract}
We study constant mean curvature hypersurfaces in $\Hyp^2\times\Hyp^2$
invariant under a double horocyclic action. We show that the CMC condition
reduces to a single autonomous ordinary differential equation for an angular
function. From this reduction, we obtain three distinct regimes and solve the 
ODE explicitly in each case, obtaining an
existence and uniqueness result for double horocyclic CMC hypersurfaces in $\Hyp^2\times\Hyp^2$.
Finally, we classify the equilibrium solutions and identify the corresponding
homogeneous models: $\Hyp^3$, $\Hyp^2\times\R$, $\mathrm{Sol}_3$, and
left-invariant metrics on semidirect product Lie groups.
\end{abstract}

\maketitle
%\tableofcontents

\section{Introduction}
\label{sec:introducao}

The study of submanifolds in product spaces of dimension four has received increasing attention in differential geometry, especially in the products $\mathbb S^2\times\mathbb S^2$ and $\Hyp^2\times\Hyp^2$. These spaces present a geometry rigid enough to allow classification results, but also rich enough to produce phenomena that do not appear in space forms. In this context, the work of Torralbo and Urbano on surfaces with parallel mean curvature vector stands out \cite{TorralboUrbano2012}; see also the recent work of Stas and Van der Veken on spheres with parallel mean curvature in product spaces \cite{StasVanDerVeken2025}. Related recent developments include studies on Hopf hypersurfaces, curvature-adapted hypersurfaces, and parallel normal Jacobi operators in products of surfaces \cite{HuLuYaoZhang2023,HuLuYaoZhangIJM2024,HuLu2024}.

In the specific case of $\Hyp^2\times\Hyp^2$, several special classes of
submanifolds have been investigated. Gao, Van der Veken, Wijffels, and Xu
studied Lagrangian surfaces in this product \cite{GaoVdVWijffelsXu2021},
while Gao, Ma, and Yao analyzed hypersurfaces of
$\Hyp^2\times\Hyp^2$ under natural geometric hypotheses
\cite{GaoMaYao2022}. More recently, Li, Vrancken, Wang, and Yao studied
hypersurfaces with constant sectional curvature
\cite{LiVranckenWangYao2026}. These works show that the geometry of
$\Hyp^2\times\Hyp^2$ is strongly influenced by the interaction between the
two hyperbolic factors.

In this paper, we adopt a complementary approach: instead of directly
imposing global conditions on sectional curvature or curvature adaptation,
we consider constant mean curvature hypersurfaces invariant under a double
horocyclic action. This symmetry reduces the problem of existence
of CMC hypersurfaces to finding a suitable
generating curve in the orbit space. We show that the CMC condition
translates into a single autonomous ordinary differential equation for an
angular function associated with the generating curve.

From this reduction, we obtain two first integrals that allow us to
explicitly reconstruct the generating curve. Next, we solve the angular
equation by a Riccati reduction and describe the trichotomy into the
subcritical, critical, and supercritical regimes. Corollary~\ref{corMainthm}
summarizes the main existence result for double horocyclic CMC hypersurfaces.
Finally, we analyze the equilibrium solutions of the angular equation and
identify the corresponding homogeneous models: $\Hyp^3$,
$\Hyp^2\times\R$, $\mathrm{Sol}_3$, and left-invariant metrics on
semidirect product Lie groups.

\section{Reduction by horocyclic symmetry}
\label{sec:setup}

\subsection{The ambient space}

We use the upper half-plane model 
\[
\Hyp^2=\{(x,y)\in\R^2:y>0\},
\qquad 
ds_{\Hyp^2}^2=\frac{dx^2+dy^2}{y^2},
\]
for the hyperbolic space of constant curvature $-1$. Then, the ambient space
we work in is the Riemannian product
\[
M=\Hyp^2\times\Hyp^2=\{(x_1,y_1,x_2,y_2):y_1,y_2>0\},
\]
\[
ds_M^2=\frac{dx_1^2+dy_1^2}{y_1^2}+\frac{dx_2^2+dy_2^2}{y_2^2}.
\]
We adopt the global orthonormal frame
\begin{equation}\label{eq:frame}
E_1=y_1\partial_{x_1},\qquad E_2=y_1\partial_{y_1},
\qquad
E_3=y_2\partial_{x_2},\qquad E_4=y_2\partial_{y_2}.
\end{equation}
Since the metric is a product, the Levi-Civita connection acts independently 
on each factor. With respect to the orthonormal frame \eqref{eq:frame}, we have
\begin{equation}\label{eq:ambient-connection}
\begin{array}{c@{\qquad\qquad}c}
\begin{aligned}
\nabla_{E_1}E_1&=E_2, &
\nabla_{E_1}E_2&=-E_1,\\
\end{aligned}
&
\begin{aligned}
\nabla_{E_3}E_3&=E_4, &
\nabla_{E_3}E_4&=-E_3,\\
\end{aligned}
\end{array}
\end{equation}
and $\nabla_{E_i}E_j=0 $, otherwise.
\subsection{Double horocyclic hypersurfaces}

In the upper half-plane model, the horocycles of $\Hyp^2$
centered at the ideal point $\infty$ are the horizontal
lines $y=\mathrm{const}$. They are preserved by the
parabolic group of horizontal translations
\begin{align*}
(x,y)\longmapsto (x+c,y),
\end{align*}
whose infinitesimal generator is the Killing vector field
\begin{align*}
K=\partial_x.
\end{align*}
In the product $M=\Hyp^2\times\Hyp^2$, we consider the two
parabolic vector fields
\begin{align*}
K_1=\partial_{x_1},\qquad K_2=\partial_{x_2},
\end{align*}
acting independently on the horizontal coordinates of the
two factors. The corresponding action of $\R^2$ is given
by
\begin{align}\label{eqaction}
\Phi_{(u,v)}(x_1,y_1,x_2,y_2) = (x_1+u,y_1,x_2+v,y_2).
\end{align}

Let $\Sigma$ be a hypersurface of $\Hyp^2\times\Hyp^2$, invariant under this double horocyclic action.
Then, $\Sigma$ is obtained by
applying the horocyclic orbits to a generating curve
transversal to the orbits. Up to a horizontal translation
in each factor, we can choose this curve in the slice
$\{x_1=x_2=0\}$, writing
\begin{align*}
\gamma(t)=(0,y(t),0,w(t)),
\qquad y(t)>0,\quad w(t)>0.
\end{align*}
Henceforth, we will omit the zeroes in the above expression
and simply state that a hypersurface $\Sigma$ with {\em generating curve}
$(y(t),w(t))$ is the surface parameterized by
\begin{equation}\label{eq:parametrizacao}
X(u,v,t) = \Phi_{(u,v)}(\gamma(t)) = (u,y(t),v,w(t)),
\qquad (u,v,t)\in\R^2\times I.
\end{equation}
We may choose the parameter $t$ so that the generating curve is
parameterized by arc length in the orbit space. Since in $\{x_1=x_2=0\}$
the induced metric is
\begin{align*}
\frac{dy^2}{y^2}+\frac{dw^2}{w^2},
\end{align*}
this normalization is equivalent to
\begin{equation}\label{eq:arco-log}
\left(\frac{y'}{y}\right)^2+\left(\frac{w'}{w}\right)^2=1.
\end{equation}

By condition \eqref{eq:arco-log}, there exists a function
$\theta=\theta(t)$, unique modulo $2\pi$, such that
\begin{equation}\label{eq:theta-def}
\frac{y'}{y}=\cos\theta,
\qquad
\frac{w'}{w}=\sin\theta,
\end{equation}
and a direct computation using~\eqref{eq:parametrizacao}
and~\eqref{eq:theta-def} shows that the induced metric on $\Sigma$ is
\begin{equation}\label{eq:metrica-induzida}
ds_\Sigma^2=\frac{du^2}{y(t)^2}+\frac{dv^2}{w(t)^2}+dt^2.
\end{equation}

\subsection{Adapted orthonormal frame and principal curvatures}

To describe the second fundamental form of $\Sigma$, 
we adapt the ambient frame to the hypersurface via a rotation by an angle~$\theta$ in the $\{E_2,E_4\}$ plane.
More precisely, if $\theta\colon I \to \R$ is the function defined by~\eqref{eq:theta-def}, then we let
\begin{equation}\label{eq:darboux}
e_1=E_1,
\qquad
e_2=E_3,
\qquad
e_3=\cos\theta\,E_2+\sin\theta\,E_4,
\qquad
N=-\sin\theta\,E_2+\cos\theta\,E_4.
\end{equation}
By~\eqref{eq:theta-def}, $e_3=\frac{\partial X}{\partial t}$, so $\{e_1,e_2,e_3\}$ is an orthonormal frame tangent to~$\Sigma$, thus $N$ is the corresponding unit normal.

\begin{proposition}[Principal curvatures]\label{prop:curvaturas-principais}
With the normal $N$ from~\eqref{eq:darboux}, the adapted frame diagonalizes the shape operator $S=-\n N$, and the principal curvatures of~$\Sigma$ are
\begin{equation}\label{eq:k123}
k_1=-\sin\theta,
\qquad
k_2=\cos\theta,
\qquad
k_3=\theta'.
\end{equation}
\end{proposition}

\begin{proof}
Since $\theta=\theta(t)$ does not vary in the directions $E_1,E_3$, using~\eqref{eq:ambient-connection} and the orthogonality of the factors we obtain
\begin{align*}
\n_{e_1}N=-\sin\theta\,\n_{E_1}E_2+\cos\theta\,\n_{E_1}E_4
=-\sin\theta\,(-E_1)=\sin\theta\,E_1,
\end{align*}
thus $S(e_1)=-\n_{e_1}N=-\sin\theta\,e_1$. Analogously,
\begin{align*}
\n_{e_2}N=-\sin\theta\,\n_{E_3}E_2+\cos\theta\,\n_{E_3}E_4
=\cos\theta\,(-E_3)=-\cos\theta\,E_3,
\end{align*}
hence $S(e_2)=\cos\theta\,e_2$. Finally, since $e_3=\cos\theta\,E_2+\sin\theta\,E_4 = \frac{\partial X}{\partial_t}$ and the terms $\n_{e_3}E_2,\n_{e_3}E_4$ vanish by~\eqref{eq:ambient-connection},
\begin{align*}
\n_{e_3}N=-\theta'\cos\theta\,E_2-\theta'\sin\theta\,E_4=-\theta'\,e_3,
\end{align*}
so that $S(e_3)=\theta'\,e_3$.
\end{proof}

\section{The CMC equation and first integrals}
\label{sec:integrabilidade}

The constant mean curvature condition can now be written directly
in terms of the angular function $\theta$. Indeed, when oriented with respect to the vector field
$N$ defined in~\eqref{eq:darboux}, Proposition~\ref{prop:curvaturas-principais} gives that
the mean curvature of $\S$ is
\begin{align*}
H_\S=\frac{-\sin\theta+\cos\theta+\theta'}{3},
\end{align*}
so $\S$ has constant mean curvature $H_\S\equiv H$ if and only if $\theta$ satisfies
\begin{equation}\label{eq:edo-cmc}
\boxed{\;
\theta'=\F(\theta):=\sin\theta-\cos\theta+\C,
\qquad
\C=3H.
\;}
\end{equation}
Using the identity
\begin{align*}
\sin\theta-\cos\theta
=
\sqrt2\,\sin\!\left(\theta-\frac\pi4\right),
\end{align*}
we can write
\begin{equation}\label{eq:F-def}
\F(\theta)
=
\sqrt2\,\sin\!\left(\theta-\frac\pi4\right)+\C.
\end{equation}

Thus, the CMC condition reduces the geometry of the family to a single autonomous ODE
for $\theta$. The next step is to show that
the generating curve $(y,w)$ can be directly reconstructed by first integrals.

\begin{theorem}[First integrals and universal reconstruction]
\label{thm:integrais-reconstrucao}
In every interval where $\F(\theta)\neq0$, the quantities
\begin{align}\label{eq:J1J2}
\J_1&=\frac{w}{y}\,e^{\C t-\theta},&
\J_2&=\frac{yw}{|\F(\theta)|}=\frac{yw}{|\theta'|}
\end{align}
are constant along any CMC solution. Furthermore, the first integrals $\J_1$ and $\J_2$ are functionally independent in the state space.
Consequently, if $\theta$ is a solution 
of~\eqref{eq:edo-cmc} with $\F(\theta)\neq0$, then
\begin{equation}\label{eq:yw-univ}
\boxed{\;
\begin{aligned}
y(t)&=\sqrt{\frac AB}\,\sqrt{|\F(\theta(t))|}\,
\exp\!\left[-\frac12\bigl(\theta(t)-\C t\bigr)\right]\\[2pt]
w(t)&=\sqrt{AB}\,\sqrt{|\F(\theta(t))|}\,
\exp\!\left[\frac12\bigl(\theta(t)-\C t\bigr)\right],
\end{aligned}
\;}
\end{equation}
is a solution $(y,w)$ of~\eqref{eq:theta-def} where
$A=\J_2$ and $B=\J_1$.
%, that is,
%\begin{equation}\label{eq:prod-razao}
%yw=A\,|\F(\theta)|,
%\qquad
%\frac wy=B\,e^{\theta-\C t}.
%\end{equation}
\end{theorem}

\begin{proof}
Let $\J_1$ and $\J_2$ be defined by~\eqref{eq:J1J2}.
Taking the logarithmic derivative and using~\eqref{eq:theta-def},
\eqref{eq:edo-cmc} and $\F'(\theta)=\cos\theta+\sin\theta$, we obtain
\begin{align*}
(\log\J_1)'
=
\frac{w'}w-\frac{y'}y+\C-\theta'
=
\sin\theta-\cos\theta+\C-\theta'
=
\F(\theta)-\theta'
=
0,
\end{align*}
and
\begin{align*}
(\log\J_2)'
=
\frac{y'}y+\frac{w'}w-\frac{\F'(\theta)\theta'}{\F(\theta)}
=
\cos\theta+\sin\theta-\F'(\theta)
=
0.
\end{align*}
Thus, $\J_1$ and $\J_2$ are constant.

For functional independence, let $\rho=\log y$ and $\sigma=\log w$.
In the state space with coordinates $(\rho,\sigma,\theta,t)$,
\begin{align*}
d\log\J_1
=
-d\rho+d\sigma-d\theta+\C\,dt,
\qquad
d\log\J_2
=
d\rho+d\sigma-\frac{\F'(\theta)}{\F(\theta)}\,d\theta.
\end{align*}
The components in $(d\rho,d\sigma)$ are $(-1,1)$ and $(1,1)$, whose
determinant is $-2\neq0$. Therefore, the differentials are linearly
independent.

Finally, setting $A=\J_2$ and $B=\J_1$, we obtain
\begin{align*}
yw=A|\F(\theta)|,
\qquad
\frac wy=B e^{\theta-\C t}.
\end{align*}
Multiplying and dividing these two identities, the formulas
\eqref{eq:yw-univ} follow.
\end{proof}
\section{The angular equation and the trichotomy}
\label{sec:fase}

By the universal reconstruction, the geometry of the double horocyclic
hypersurface reduces to the autonomous angular equation
\begin{equation}\label{eq:mestra}
\theta'=\F(\theta)
=
\sqrt2\,\sin\!\left(\theta-\frac\pi4\right)+\C.
\end{equation}
In this section, we directly analyze the phase line of~\eqref{eq:mestra}. The
position of the zeros of $\F$ determines the equilibria of the angular dynamics
and, consequently, separates the solutions into three qualitatively distinct
regimes, which will be analyzed in the next section.

\begin{remark}[Symmetry in $\C$]\label{rmk:simetria}
The involution
\begin{align*}
(\theta,\C)\longmapsto \left(\frac\pi2-\theta,-\C\right)
\end{align*}
preserves equation~\eqref{eq:mestra}. Furthermore, it swaps the equations
\begin{align*}
\frac{y'}y=\cos\theta,
\qquad
\frac{w'}w=\sin\theta
\end{align*}
with each other, that is, it geometrically corresponds to swapping the two
factors $y\leftrightarrow w$. Therefore, up to this symmetry, it suffices to
analyze the case $\C\ge0$.
\end{remark}

\subsection{The phase line}
\label{subsec:tricotomia} 

The equilibria of~\eqref{eq:mestra} are the zeros of 
$\F$, that is, the constant
values $\theta_\C$ that satisfy 
\begin{equation}\label{eq:zeros-F}
\sin\!\left(\theta_\C-\frac\pi4\right)
=
-\frac{\C}{\sqrt2}.
\end{equation}
It follows immediately that equilibria exist if and only if
$|\C|\le\sqrt2$. Thus, the critical value $\sqrt2$ separates three
possibilities, governed by the position of $\C$ 
relative to $\sqrt2$ (see Figure~\ref{fig:tricotomia}):

\addvspace{\topsep}

\noindent {\bf Trichotomy.}
The dynamics of $\theta'=\F(\theta)$ is 
\begin{itemize}
\item {\bf Subcritical} ($0\le\C<\sqrt2$): two simple equilibria modulo $2\pi$, a 
repeller and an attractor.
\item {\bf Critical} ($\C=\sqrt2$): a 
single degenerate equilibrium point modulo $2\pi$, corresponding to a saddle-node 
bifurcation.
\item {\bf Supercritical} ($\C>\sqrt2$): there are no equilibria; 
the angular velocity is always positive, and $\theta$ traverses the 
circle.
\end{itemize}

\addvspace{\topsep}

\begin{figure}[ht]
\centering
\includegraphics[width=0.86\textwidth]{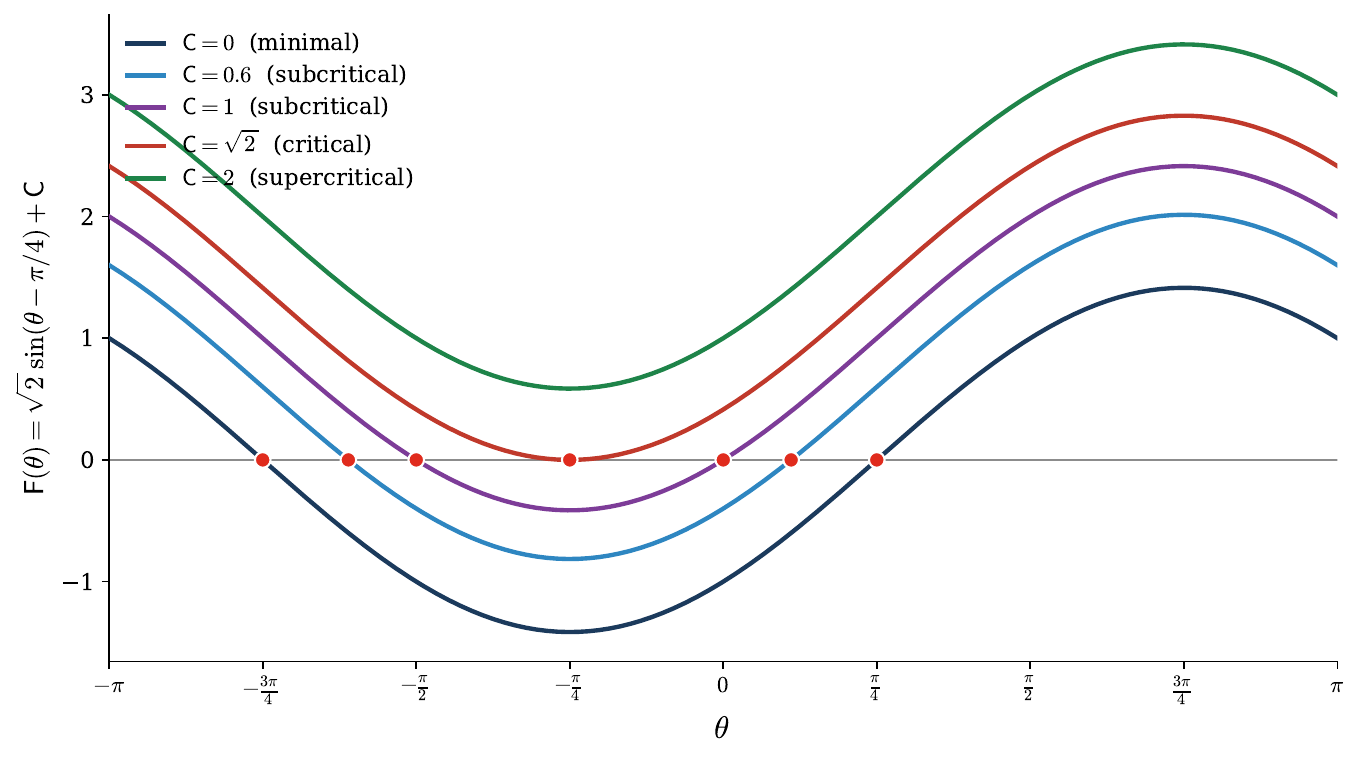}
\caption{The graph of $\F(\theta)=\sqrt2\sin(\theta-\pi/4)+\C$ in 
representative regimes. When $0\leq \C <\sqrt{2}$, there are two equilibria points (subcritical regime); when $\C = \sqrt{2}$ (critical regime) there is only one equilibrium point; when $\C>\sqrt{2}$ (supercritical regime), there are no equilibria points.}
\label{fig:tricotomia}
\end{figure}

In the subcritical regime, $0\le\C<\sqrt2$, the function $\F$ has 
two simple zeros $\theta_\C^-,\,\theta_\C^+$ modulo $2\pi$. 
For convenience, we choose
\begin{align*}
-3\pi/4\leq \theta_\C^-<\theta_\C^+\leq \pi/4,
\end{align*}
so $\theta_\C^-$ and $\theta_\C^+$ relate via
\begin{align}\label{eqrelatethc}
\theta_\C^++\theta_\C^- = -\pi/2.
\end{align}
Thus, after a translation of the initial condition, 
any non-constant solution of $\theta'=\F(\theta)$ must 
satisfy either $\theta_\C^- < \theta(t) < \theta_\C^+$, being
strictly decreasing since $\theta'(t)<0$, or
$\theta_\C^+<\theta(t) < \theta_\C^-+2\pi$ being strictly increasing.
Therefore, the phase line has two equilibria modulo $2\pi$:
$\theta_\C^+$ is a repeller and $\theta_\C^-$ is an attractor.
Every nonconstant solution is heteroclinic, modulo $2\pi$, from the
repelling equilibrium to the attracting one.
As $\C$ increases, these two equilibria
approach each other until they coalesce at the critical value $\C=\sqrt2$.

In the critical regime, $\C=\sqrt2$, the two zeros of the subcritical regime 
become a single double zero, given by
\begin{align*}
\theta_\C=-\frac\pi4 \mod 2\pi.
\end{align*}
The phase line becomes tangent to the axis, and the 
resulting equilibrium is 
degenerate. Any nonconstant solution is increasing and bounded 
between $-\pi/4$ and $7\pi/4$.

In the supercritical regime, $\C>\sqrt2$, the function $\F$ is strictly 
positive, since
\begin{align*}
\F(\theta)\ge \C-\sqrt2>0.
\end{align*}
Thus, there are no equilibria. 
The function $\theta$ is strictly increasing 
and traverses the angular circle repeatedly.

\subsection{Equilibrium solutions.} Next, we will provide explicit
expressions for the equilibrium solutions in the subcritical and critical 
regime. The solutions that correspond to nonconstant $\theta$ functions 
will be analyzed in Section~\ref{sec:solucao}.

\begin{proposition}[Equilibrium solutions]\label{propEqSol}
Let $\S$ be a hypersurface in $\Hyp^2\times\Hyp^2$ with constant mean
curvature $H\geq 0$
parameterized by 
\begin{align*}
X(u,v,t) = (u,y(t),v,w(t)),
\end{align*}
where the generating curve $(y(t),w(t))$ corresponds to
an equilibrium solution of~\eqref{eq:mestra}. Then $H\in [0,\sqrt{2}/3]$
and we have one of the following
solutions, for $y_0,w_0>0$:
\begin{equation}\label{eq:equilsol}
\boxed{\;
y(t) = y_0e^{\cos(\theta_\C^+)t},\quad 
w(t) = w_0e^{\sin(\theta_\C^+)t}\;}
\quad
\boxed{\;
y(t) = y_0e^{-\sin(\theta_\C^+)t},\quad
w(t) = w_0e^{-\cos(\theta_\C^+)t},\;}
\end{equation}
where $\theta_\C^+$ is the unique solution to~\eqref{eq:zeros-F}
in $[-\pi/4,\pi/4]$, for $\C = 3H$.
\end{proposition}
\begin{proof}
As previously observed, equilibrium solutions appear only for
$3H = \C\in[0,\sqrt{2}]$, which correspond to 
$H\in[0,\sqrt{2}/3]$. In this 
case, if $\C<\sqrt{2}$, there exist two solutions 
$\theta_\C^-,\,\theta_\C^+\in [-3\pi/4,\pi/4]$
for~\eqref{eq:zeros-F}, which relate by~\eqref{eqrelatethc}.
When $H = \sqrt{2}/3$, then $\C = \sqrt{2}$, in which case
$\theta^+_{\sqrt{2}} = \theta^-_{\sqrt{2}} = -\pi/4$.

The rest of the proof follows directly from~\eqref{eq:theta-def},
after observing that (also in the critical regime)
\begin{align*}
\cos(\theta_\C^+)& = -\sin(\theta_\C^-),&
\sin(\theta_\C^+)& = -\cos(\theta_\C^-).\qedhere
\end{align*}
\end{proof}

\section{The Riccati equation and explicit solutions}
\label{sec:solucao}

In this section, we use the classical Weierstrass substitution to transform
the angular equation into a Riccati equation, a standard form in the theory of
ordinary differential equations; see, for instance, Ince~\cite[pp.~23--25]{Ince1956}.

\begin{lemma}[Change of variables]\label{lem:weierstrass}
The substitution
\begin{equation}\label{eq:u-def}
u(t):=\tan\!\left(\frac{\theta(t)-\pi/4}{2}\right)
\quad\Longleftrightarrow\quad
\theta(t)=2\arctan u(t)+\frac\pi4,
\end{equation}
transforms the CMC equation~\eqref{eq:edo-cmc} into the polynomial Riccati
equation with constant coefficients
\begin{equation}\label{eq:riccati}
2u'(t)=\C u(t)^2+2\sqrt2\,u(t)+\C.
\end{equation}
When $\C>0$, multiplying by $\C$ and rearranging, we obtain the factored form
\begin{equation}\label{eq:vform}
2\C u'(t)=\bigl(\C u(t)+\sqrt2\bigr)^2-(2-\C^2).
\end{equation}
\end{lemma}

\begin{proof}
From the definition, $\theta'=2u'/(1+u^2)$. Using $\sin(2\arctan u)=2u/(1+u^2)$,
\begin{align*}
\F(\theta)=\sqrt2\,\sin(\theta-\pi/4)+\C
=\frac{2\sqrt2\,u}{1+u^2}+\C
=\frac{2\sqrt2\,u+\C(1+u^2)}{1+u^2}.
\end{align*}
Equating to $\theta'=2u'/(1+u^2)$,
\begin{align*}
2u'=\C u^2+2\sqrt2\,u+\C,
\end{align*}
which is~\eqref{eq:riccati}. For~\eqref{eq:vform}, we multiply both sides
of~\eqref{eq:riccati} by $\C$:
\begin{align*}
2\C u'=\C^2u^2+2\sqrt2\,\C u+\C^2=(\C u)^2+2\sqrt2\,(\C u)+\C^2.
\end{align*}
Completing the square in $\C u$,
\begin{align*}
(\C u)^2+2\sqrt2\,(\C u)+\C^2=(\C u+\sqrt2)^2-2+\C^2=(\C u+\sqrt2)^2-(2-\C^2).\,\,\,\,\,\,\qedhere
\end{align*}
\end{proof}

\subsection[Minimal hypersurface case]{Minimal hypersurface case ($\C=0$)}
\label{subsec:minimo}

When $\C=0$, the Riccati equation 
reduces to the linear equation
$u'=\sqrt2\,u$. Thus, $u\equiv 0$ is the unique equilibrium solution 
(which yields $\theta(t)\equiv \pi/4 \mod \pi$ as previously observed), and
any nonconstant solution is
\begin{align*}
u(t)=a\,e^{\sqrt2 t}, \qquad a>0.
\end{align*}
Since a translation of the parameter $t$ allows us to absorb the constant 
$a$, we can take, without loss of generality,
\begin{align*}
u(t)=e^{\sqrt2 t}.
\end{align*}
Therefore,
\begin{equation}\label{eq:min-theta}
\theta(t)=\frac\pi4+2\arctan(e^{\sqrt2 t}),
\qquad
\F(\theta(t))=\sqrt2\,\sech(\sqrt2\,t).
\end{equation}
Thus, Theorem~\ref{thm:integrais-reconstrucao} gives that the 
solutions corresponding to nonconstant $\theta$ function are
\begin{equation}\label{eq:min-yw}
\boxed{\;
\begin{aligned}
y(t)
&=
\sqrt{\frac{\sqrt2\,A}{B}}\,
\sqrt{\sech(\sqrt2\,t)}\,
\exp\!\left[
-\arctan(e^{\sqrt2 t})-\frac{\pi}{8}
\right],
\\[2pt]
w(t)
&=
\sqrt{\sqrt2\,AB}\,
\sqrt{\sech(\sqrt2\,t)}\,
\exp\!\left[
\arctan(e^{\sqrt2 t})+\frac{\pi}{8}
\right].
\end{aligned}
\;}
\end{equation}

 In Corollary~\ref{corMainthm} we
will obtain a normalization that will further
simplify~\eqref{eq:min-yw}. The normalized representative displayed in Figure~\ref{fig:curvas} is a
symmetric loop under $(y,w)\leftrightarrow(w,y)$.

In the explicit solutions that will be presented for the next regimes, 
a translation constant $t_0$ appears in 
general. Since the angular equation is autonomous, this constant merely 
shifts the origin of the parameter $t$. Therefore, in all non-constant 
regimes, we will set $t_0=0$ without loss of generality.

\subsection[Subcritical case]{Subcritical case ($0<\C<\sqrt{2}$)}
\label{subsec:subcritico}

Assume that $0<\C<\sqrt{2}$ and set
\[
\alpha:=\sqrt{2-\C^{2}}.
\]
Introduce the shifted Riccati variable
\[
v(t):=\C u(t)+\sqrt{2}.
\]
Then the factored Riccati equation~\eqref{eq:vform} becomes
\begin{equation}\label{eq:sub-vform}
2v'=v^{2}-\alpha^{2}.
\end{equation}
Equivalently,
\begin{equation}\label{eq:sub-vform-half}
v'=\frac{1}{2}\bigl(v^{2}-\alpha^{2}\bigr).
\end{equation}

We first solve this equation on the bounded region $|v|<\alpha$. In this
region we have $v^{2}-\alpha^{2}=-(\alpha^{2}-v^{2})$, and therefore
\[
v'=-\frac{1}{2}\bigl(\alpha^{2}-v^{2}\bigr).
\]
Separating variables gives
\begin{equation}\label{eq:sub-sep-tanh}
\frac{dv}{\alpha^{2}-v^{2}}
=
-\frac{1}{2}\,dt.
\end{equation}
Hence
\begin{equation}\label{eq:sub-int-tanh}
\frac{1}{\alpha}
\operatorname{arctanh}\left(\frac{v}{\alpha}\right)
=
-\frac{1}{2}(t-t_{0}).
\end{equation}
Thus
\[
\operatorname{arctanh}\left(\frac{v}{\alpha}\right)
=
-\frac{\alpha}{2}(t-t_{0}),
\]
and consequently
\[
v(t)
=
-\alpha
\tanh\left(\frac{\alpha}{2}(t-t_{0})\right).
\]
Since the equation is autonomous, the constant $t_{0}$ can be absorbed by a
translation of the parameter $t$. Therefore we may assume $t_{0}=0$ and obtain
\begin{equation}\label{eq:v-tanh}
v(t)
=
-\alpha
\tanh\left(\frac{\alpha t}{2}\right).
\end{equation}

Returning to the variable $u$, the relation $v=\C u+\sqrt{2}$ gives
\[
\C u=v-\sqrt{2}.
\]
Using~\eqref{eq:v-tanh}, we obtain
\begin{equation}\label{eq:sub-u-tanh}
\boxed{\;
\C u(t)
=
-\sqrt{2}
-
\alpha
\tanh\left(\frac{\alpha t}{2}\right),
\qquad
0<\C<\sqrt{2}.
\;}
\end{equation}
Equivalently,
\begin{equation}\label{eq:u-tanh-explicit}
u(t)
=
\frac{
-\sqrt{2}
-
\alpha
\tanh\left(\frac{\alpha t}{2}\right)
}{\C}.
\end{equation}
Therefore the corresponding angular function is
\begin{equation}\label{eq:sub-theta-tanh}
\theta(t)
=
\frac{\pi}{4}
+
2\arctan\left(
\frac{
-\sqrt{2}
-
\alpha
\tanh\left(\frac{\alpha t}{2}\right)
}{\C}
\right).
\end{equation}
This solution satisfies
\[
\theta_{\C}^{-}<\theta(t)<\theta_{\C}^{+}
\]
and is strictly decreasing. Hence it is the heteroclinic branch from the
repelling equilibrium $\theta_{\C}^{+}$ to the attracting equilibrium
$\theta_{\C}^{-}$.

We now reconstruct the generating functions $y$ and $w$. Since
\[
u=\frac{v-\sqrt{2}}{\C},
\qquad
1+u^{2}
=
\frac{\C^{2}+\bigl(v-\sqrt{2}\bigr)^{2}}{\C^{2}},
\]
and since
\[
\theta'=\frac{2u'}{1+u^{2}},
\qquad
v'=\C u',
\]
we get
\begin{equation}\label{eq:F-in-v}
\F(\theta)
=
\frac{
\C\bigl(v^{2}-\alpha^{2}\bigr)
}
{
\C^{2}+\bigl(v-\sqrt{2}\bigr)^{2}
}.
\end{equation}
For the solution~\eqref{eq:v-tanh}, we have
\[
v(t)^{2}-\alpha^{2}
=
-\alpha^{2}\sech^{2}\left(\frac{\alpha t}{2}\right)
\]
and
\[
v(t)-\sqrt{2}
=
-\sqrt{2}
-
\alpha
\tanh\left(\frac{\alpha t}{2}\right).
\]
Define
\begin{equation}\label{eq:D-tanh}
D_{\mathrm{tanh}}(t)
:=
\C^{2}
+
\left(
\sqrt{2}
+
\alpha
\tanh\left(\frac{\alpha t}{2}\right)
\right)^{2}.
\end{equation}
Then
\begin{equation}\label{eq:sub-F-tanh}
|\F(\theta(t))|
=
\frac{
\C\alpha^{2}\sech^{2}\left(\frac{\alpha t}{2}\right)
}
{
D_{\mathrm{tanh}}(t)
}.
\end{equation}
By the universal reconstruction formula~\eqref{eq:yw-univ}, the bounded
subcritical branch is
\begin{equation}\label{eq:sub-yw-tanh}
\boxed{\;
\begin{aligned}
y(t)
&=
\sqrt{\frac{A}{B}}\,
\frac{
\alpha\sqrt{\C}
\sech\left(\frac{\alpha t}{2}\right)}
{
\sqrt{D_{\mathrm{tanh}}(t)}
}
\exp\left[
-\frac{1}{2}\bigl(\theta(t)-\C t\bigr)
\right],
\\[2pt]
w(t)
&=
\sqrt{AB}\,
\frac{
\alpha\sqrt{\C}
\sech\left(\frac{\alpha t}{2}\right)}
{
\sqrt{D_{\mathrm{tanh}}(t)}
}
\exp\left[
\frac{1}{2}\bigl(\theta(t)-\C t\bigr)
\right].
\end{aligned}
\;}
\end{equation}
Here $\theta(t)$ is the function defined in~\eqref{eq:sub-theta-tanh};
the corresponding generating curve is shown in Figure~\ref{fig:curvas}.

We now describe the second possible nonconstant branch. This branch corresponds
to the unbounded solutions of~\eqref{eq:sub-vform}, namely the region
$|v|>\alpha$. In this region we separate variables as
\begin{equation}\label{eq:sub-sep-coth}
\frac{dv}{v^{2}-\alpha^{2}}
=
\frac{1}{2}\,dt.
\end{equation}
Integrating and translating the parameter $t$, we obtain
\begin{equation}\label{eq:v-coth}
v(t)
=
-\alpha
\operatorname{coth}\left(\frac{\alpha t}{2}\right),
\qquad
t\neq 0.
\end{equation}
Returning to $u$, we get
\begin{equation}\label{eq:sub-u-coth}
\C u(t)
=
-\sqrt{2}
-
\alpha
\operatorname{coth}\left(\frac{\alpha t}{2}\right).
\end{equation}
Thus
\begin{equation}\label{eq:u-coth-explicit}
u(t)
=
\frac{
-\sqrt{2}
-
\alpha
\operatorname{coth}\left(\frac{\alpha t}{2}\right)
}{\C}.
\end{equation}
The coordinate $u$ has a simple pole at $t=0$, but this is only a singularity of the Weierstrass substitution. To obtain a continuous lift of the angular variable, we define \[ \theta(t)= \begin{cases} \dfrac{\pi}{4}+2\arctan u(t), & t<0,\\[6pt] \dfrac{5\pi}{4}, & t=0,\\[6pt] \dfrac{9\pi}{4}+2\arctan u(t), & t>0. \end{cases} \] Indeed, the one-sided limits of the principal expression $\frac{\pi}{4}+2\arctan u(t)$ are $\frac{5\pi}{4}$ as $t\to0^-$ and $-\frac{3\pi}{4}$ as $t\to0^+$, which agree modulo $2\pi$.
This branch satisfies
\[
\theta_{\C}^{+}<\theta(t)<\theta_{\C}^{-}+2\pi
\]
and is strictly increasing. Hence it travels from $\theta_{\C}^{+}$ to
$\theta_{\C}^{-}+2\pi$ along the complementary component of the angular circle.

For this branch, define
\begin{equation}\label{eq:D-coth}
D_{\mathrm{coth}}(t)
:=
\C^{2}
+
\left(
\sqrt{2}
+
\alpha
\operatorname{coth}\left(\frac{\alpha t}{2}\right)
\right)^{2}.
\end{equation}
Then
\begin{equation}\label{eq:sub-F-coth}
|\F(\theta(t))|
=
\frac{
\C\alpha^{2}
\csch^{2}\left(\frac{\alpha t}{2}\right)}
{
D_{\mathrm{coth}}(t)
}.
\end{equation}
Again by the universal reconstruction formula~\eqref{eq:yw-univ}, the
generating functions for the unbounded branch are
\begin{equation}\label{eq:sub-yw-coth}
\boxed{\;
\begin{aligned}
y(t)
&=
\sqrt{\frac{A}{B}}\,
\frac{
\alpha\sqrt{\C}\left|\csch\left(\frac{\alpha t}{2}\right)\right|
}
{
\sqrt{D_{\mathrm{coth}}(t)}
}
\exp\left[
-\frac{1}{2}\bigl(\theta(t)-\C t\bigr)
\right],
\\[2pt]
w(t)
&=
\sqrt{AB}\,
\frac{
\alpha\sqrt{\C}
\left|\csch\left(\frac{\alpha t}{2}\right)\right|}
{
\sqrt{D_{\mathrm{coth}}(t)}
}
\exp\left[
\frac{1}{2}\bigl(\theta(t)-\C t\bigr)
\right].
\end{aligned}
\;}
\end{equation}
Although the Riccati coordinate has a pole at $t=0$, the lifted angular
function and the reconstructed generating curve extend across this value.
This second branch is the complementary subcritical branch; the corresponding
generating curve is shown in Figure~\ref{fig:curvas}.

\subsection{Critical case ($\C=\sqrt2$)}
\label{subsec:critico}

When $\C=\sqrt2$, equation \eqref{eq:vform} becomes
\[
2v'=v^2.
\]
Its nonconstant solutions are rational in the Riccati variable. In the angular
variable, choosing the origin of $t$ conveniently, we may write the regular
branch as
\begin{equation}\label{eq:crit-theta}
\theta(t)=\frac{3\pi}{4}+2\arctan(\sqrt2\,t),
\qquad
\F(\theta(t))=\frac{2\sqrt2}{1+2t^2}.
\end{equation}
Thus, by the universal reconstruction formula \eqref{eq:yw-univ}, we obtain
\begin{equation}\label{eq:crit-yw-general}
\boxed{\;
\begin{aligned}
y(t)
&=
\sqrt{\frac AB}\,
\frac{\sqrt{2\sqrt2}}{\sqrt{1+2t^2}}\,
\exp\!\left[
-\frac{3\pi}{8}
-\arctan(\sqrt2\,t)
+\frac{\sqrt2}{2}t
\right],
\\[2pt]
w(t)
&=
\sqrt{AB}\,
\frac{\sqrt{2\sqrt2}}{\sqrt{1+2t^2}}\,
\exp\!\left[
\frac{3\pi}{8}
+\arctan(\sqrt2\,t)
-\frac{\sqrt2}{2}t
\right],
\end{aligned}
\;}
\end{equation}
where $A>0$ and $B>0$ are the integration constants appearing in
Theorem~\ref{thm:integrais-reconstrucao}. The corresponding generating curve
is shown in Figure~\ref{fig:curvas}.

\subsection{Supercritical case ($\C>\sqrt2$)}
\label{subsec:supercritico}

With $\omega=\sqrt{\C^2-2}$, \eqref{eq:vform} becomes $2v'=v^2+\omega^2$, 
which has the periodic solution $v(t)=\omega\tan(\omega t/2)$, hence
\begin{equation}\label{eq:sup-u}
\C\,u(t)=-\sqrt2+\omega\tan(\omega t/2),
\qquad
\theta(t)=\frac\pi4+2\arctan u(t).
\end{equation}
Writing $D(t)=\C^2+\big(-\sqrt2+\omega\tan(\omega t/2)\big)^2$, we have
\begin{equation}\label{eq:sup-F}
\F(\theta(t))=\frac{\C\,\omega^2\,\sec^2(\omega t/2)}{D(t)}>0,
\end{equation}
and from \eqref{eq:yw-univ} we obtain the generating functions
\eqref{eq:sup-yw-final}. The corresponding generating curve is shown in
Figure~\ref{fig:curvas}.

\begin{equation}\label{eq:sup-yw-final}
\boxed{\;
\begin{aligned}
y(t)&=\sqrt{\frac AB}\;\frac{\omega\sqrt\C\,|\!\sec(\omega t/2)|}{\sqrt{D(t)}}\,
       e^{-(\theta(t)-\C t)/2},\\[2pt]
w(t)&=\sqrt{AB}\;\frac{\omega\sqrt\C\,|\!\sec(\omega t/2)|}{\sqrt{D(t)}}\,
       e^{(\theta(t)-\C t)/2}.
\end{aligned}
\;}
\end{equation}

\subsection{Summary.}

In Figure~\ref{fig:curvas}, we visualize the generating
curve for some normalized cases. 
The analysis carried out in the four cases as above,
together with Proposition~\ref{propEqSol}, 
gives our main existence result concerning
hypersurfaces of $\Hyp^2\times\Hyp^2$ invariant under 
a double horocyclic action:

\begin{theorem}\label{mainThm}
Let $\S$ be a hypersurface of $\Hyp^2\times\Hyp^2$ which is invariant under the
double horocyclic action of~\eqref{eqaction}. Then, $\S$ has 
constant mean curvature $H\geq 0$
if and only if there exists an open interval $I$
such that $\S$ is parameterized by 
\begin{align*}
X(u,v,t) = (u,y(t),v,w(t)),\qquad (u,v,t)\in \R^2\times I,
\end{align*}
where $y(t)$ and $w(t)$ are either given by Proposition~\ref{propEqSol}
or given by~\eqref{eq:min-yw}, if $H = 0$,
by~\eqref{eq:sub-yw-tanh} or~\eqref{eq:sub-yw-coth}, if $H\in (0,\sqrt{2}/3)$, by~\eqref{eq:crit-yw-general}, if $H = \sqrt{2}/3$
and by~\eqref{eq:sup-yw-final}, if $H >\sqrt{2}/3$.
\end{theorem}

Next, we note that for any $a,b\in \R$, the map 
\begin{align*}
\psi_{a,b}(x_1,y_1,x_2,y_2) = (e^ax_1,e^ay_1,e^bx_2,e^by_2)
\end{align*}
is an isometry of $\Hyp^2\times \Hyp^2$. Indeed, it is the composition
of a hyperbolic translation in the first factor with a hyperbolic translation
in the second factor. 

If $\S$ is a hypersurface invariant under the double horocyclic
action given by~\eqref{eqaction}, parameterized as in Theorem~\ref{mainThm},
the surface $\S_{a,b} = \psi_{a,b}(\S)$ is also a hypersurface invariant
by the same action (up to a reparameterization) 
and with the same constant mean curvature. If $(y(t),w(t))$
is the generating curve of $\S$, then the generating
curve of $\S_{a,b}$ is simply $(e^ay(t),e^bw(t))$.
Hence, after appropriately choosing $a$ and $b$ to cancel out the 
integration constants $A$ and $B$ given by
Theorem~\ref{thm:integrais-reconstrucao},
we have the following normalization.

\begin{corollary}\label{corMainthm} After an ambient isometry, any hypersurface of
$\Hyp^2\times\Hyp^2$ with constant mean curvature $H\geq 0$ which is
invariant under the double horocyclic action can be parameterized as
\[
X(u,v,t)=(u,y(t),v,w(t)).
\]
Moreover, either $\theta$ is constant or one of the following nonconstant
normal forms holds.

If $\theta$ is constant, then necessarily $0\leq \C\leq\sqrt2$ and
\[
y(t)=e^{ct},
\qquad
w(t)=e^{st},
\]
where
\[
c^2+s^2=1,
\qquad
c-s=\C.
\]

If $\theta$ is nonconstant, then the following alternatives occur.

\begin{enumerate}[label=\textup{(\roman*)},leftmargin=*]

\item If $H=0$, then
\begin{equation*}
\boxed{\;
\begin{aligned}
y(t)&=e^{\pi/4}\sqrt{\sech(\sqrt2\,t)}\,\exp\!\left[-\arctan\!\left(e^{\sqrt2 t}\right)\right],\\[2pt]
w(t)&=e^{-\pi/4}\sqrt{\sech(\sqrt2\,t)}\,\exp\!\left[\arctan\!\left(e^{\sqrt2 t}\right)\right].
\end{aligned}\;}
\end{equation*}

\item If $0<H<\sqrt2/3$, then $0<\C<\sqrt2$. Writing $\alpha=\sqrt{2-\C^2}$, there
are \emph{two} nonconstant normal forms, one for each branch of
Subsection~\ref{subsec:subcritico}. After the preceding ambient normalization, the bounded ($\tanh$) branch is
\[
\boxed{\;
y(t)=
\frac{\alpha\sqrt\C\,\sech(\tfrac{\alpha t}{2})}{\sqrt{D_{\tanh}(t)}}\,
e^{-(\theta(t)-\C t)/2},
\qquad
w(t)=
\frac{\alpha\sqrt\C\,\sech(\tfrac{\alpha t}{2})}{\sqrt{D_{\tanh}(t)}}\,
e^{(\theta(t)-\C t)/2},
\;}
\]
and the unbounded ($\coth$) branch is
\[
\boxed{\;
y(t)=
\frac{\alpha\sqrt\C\,\bigl|\mathrm{csch}(\tfrac{\alpha t}{2})\bigr|}{\sqrt{D_{\coth}(t)}}\,
e^{-(\theta(t)-\C t)/2},
\qquad
w(t)=
\frac{\alpha\sqrt\C\,\bigl|\mathrm{csch}(\tfrac{\alpha t}{2})\bigr|}{\sqrt{D_{\coth}(t)}}\,
e^{(\theta(t)-\C t)/2},}
\]
where $\theta$, $D_{\tanh}$ and $D_{\coth}$ are as in
Subsection~\ref{subsec:subcritico}.
\item If $H=\sqrt2/3$, then
\begin{equation*}
\boxed{\;
\begin{aligned}
y(t)
&=
\frac{
\exp\!\bigl[-\arctan(\sqrt2\,t)+\tfrac{\sqrt2}{2}t\bigr]}
{\sqrt{1+2t^2}},
\\[2pt]
w(t)
&=
\frac{
\exp\!\bigl[\arctan(\sqrt2\,t)-\tfrac{\sqrt2}{2}t\bigr]}
{\sqrt{1+2t^2}}.
\end{aligned}
\;}
\end{equation*}

\item If $H>\sqrt2/3$, then $\C>\sqrt2$. Writing $\omega=\sqrt{\C^2-2}$, with
$A=B=1$ the normal form is
\[
\boxed{\;
y(t)=
\frac{\omega\sqrt\C\,\bigl|\sec(\tfrac{\omega t}{2})\bigr|}{\sqrt{D(t)}}\,
e^{-(\theta(t)-\C t)/2},
\qquad
w(t)=
\frac{\omega\sqrt\C\,\bigl|\sec(\tfrac{\omega t}{2})\bigr|}{\sqrt{D(t)}}\,
e^{(\theta(t)-\C t)/2},
\;}
\]
valid on any interval where $\tan(\tfrac{\omega t}{2})$ is finite, with $\theta$
and $D$ as in Subsection~\ref{subsec:supercritico}.

\end{enumerate}
\end{corollary}

\begin{figure}[ht]
\centering
\includegraphics[width=0.92\textwidth]{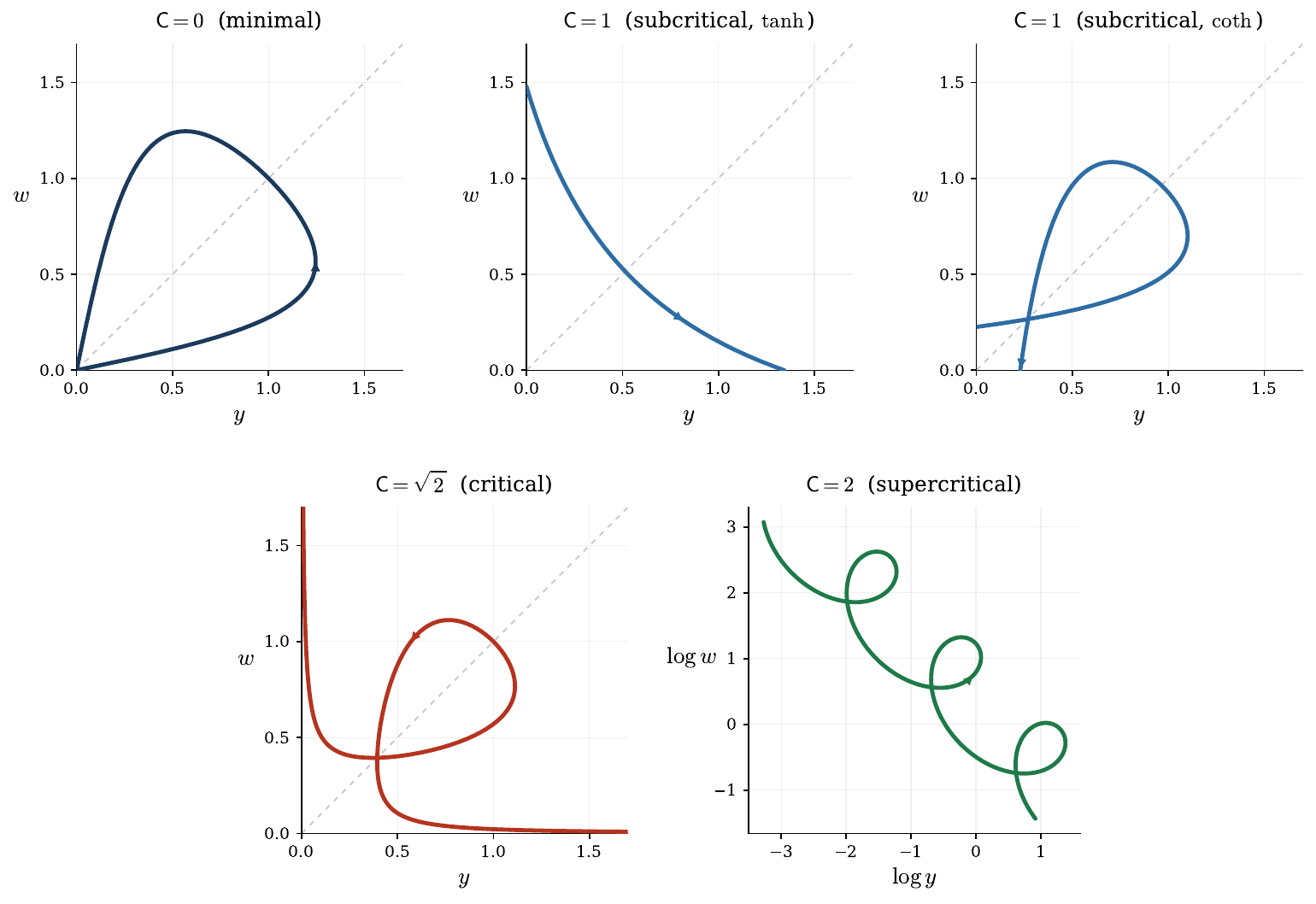}
\caption{Generating curves $(y,w)$ for representative values of $\C$, reconstructed from $\theta(t)$ by the universal formula~\eqref{eq:yw-univ} and shown up to an ambient dilation $\psi_{a,b}$; arrows indicate increasing $t$. In the subcritical regime, illustrated for $\C=1$, the bounded ($\tanh$) branch is a simple arc, whereas the unbounded ($\coth$) branch is a loop. The critical branch $\C=\sqrt{2}$ has a self-intersection. In the minimal case, the normalized generating curve is symmetric under the interchange $(y,w)\leftrightarrow(w,y)$ and both ends approach the origin. The supercritical curve is shown in logarithmic coordinates $(\log y,\log w)$, where the angular winding produces a spiral.}
\label{fig:curvas}
\end{figure}

\section{Equilibrium solutions and homogeneous geometries}
\label{sec:equilibrios}

The equilibrium solutions of the angular equation, that is, the constant
solutions $\theta(t)\equiv\theta_\C$, play a special role in the family.
Dynamically, they are the simplest solutions; geometrically, however, they
give rise to the natural homogeneous models associated with the horocyclic
reduction. Indeed, when $\theta$ is constant, \eqref{eq:equilsol} shows that
the functions $y$ and $w$ are exponential in $t$. Consequently, the induced
metric takes the form of a left-invariant metric on a three-dimensional
semidirect product Lie group.

In this section, we classify these equilibria and identify the corresponding
homogeneous geometries. The result gives an explicit dictionary between the
values of the parameter $\C$, the regimes of the angular equation, and the
three-dimensional geometries appearing in the family. Besides the special
models $\Hyp^3(-1/2)$, $\Hyp^2(-1)\times\R$, and ${\mathrm{Sol}}_3$, the generic case is described by a family of diagonal semidirect product Lie groups
\[
G_{c,s}=\R^2\rtimes_{\varphi}\R,
\qquad
\varphi_t=\mathrm{diag}(e^{ct},e^{st}),
\]
equivalently,
\[
(u,v,t)\cdot(\bar u,\bar v,\bar t)
=
\bigl(u+e^{ct}\bar u,\,
v+e^{st}\bar v,\,
t+\bar t\bigr).
\]
Here $c=\cos\theta_\C$ and $s=\sin\theta_\C$. The special choices of
$(c,s)$ recover the isotropic hyperbolic case, the product case, and the
Sol case; all remaining choices give anisotropic left-invariant
metrics.

\subsection{Equilibrium metrics}
\label{subsec:metricas-equilibrio}

For $\C\in[0,\sqrt{2}]$, let
$\theta_\C^+$ denote the unique solution to~\eqref{eq:mestra} in
$[-\pi/4,\pi/4]$ and 
$\theta_\C^- = -\theta^+_\C-\pi/2\in [-3\pi/4,-\pi/4]$. Let
$\theta_\C$ be either $\theta_\C^+$ or $\theta_\C^-$, and we write
\begin{align*}
c=\cos\theta_\C, \qquad s=\sin\theta_\C.
\end{align*}
In particular, $\C=c-s$ and we have that the generating curves to 
a hypersurface $\S$ as in Proposition~\ref{propEqSol} are
\begin{equation}\label{eq:eq-yw}
y(t)=y_0e^{ct}, \qquad w(t)=w_0e^{st}, \qquad y_0,w_0>0.
\end{equation}
After a hyperbolic dilation in each factor, as in 
Corollary~\ref{corMainthm}, we may assume that
$y_0=w_0=1$. Thus, the metric~\eqref{eq:metrica-induzida} of $\S$
becomes
\begin{equation}\label{eq:metrica-equilibrio}
\boxed{\;
ds_\Sigma^2
=e^{-2ct}\,du^2+e^{-2st}\,dv^2+dt^2.
\;}
\end{equation}

This metric is homogeneous. Indeed, consider the semidirect product Lie group
\[
G_{c,s}=\R^2\rtimes_{\varphi}\R,
\qquad
\varphi_t=\mathrm{diag}(e^{ct},e^{st}),
\]
with group law
\begin{align*}
(u,v,t)\cdot(\bar u,\bar v,\bar t)
=
\bigl(u+e^{ct}\bar u,\, v+e^{st}\bar v,\, t+\bar t\bigr).
\end{align*}
Then the vector fields
\begin{align*}
e_1=e^{ct}\partial_u, \qquad e_2=e^{st}\partial_v, \qquad e_3=\partial_t
\end{align*}
are left-invariant and orthonormal for the metric \eqref{eq:metrica-equilibrio}. 
Therefore, \eqref{eq:metrica-equilibrio} is a left-invariant metric on $G_{c,s}$.

\subsection{Identification of the homogeneous models}
\label{subsec:modelos-homogeneos}

\begin{theorem}[Homogeneous equilibrium models]
\label{thm:modelos-homogeneos}
The equilibrium metric \eqref{eq:metrica-equilibrio} belongs exactly to 
one of the classes below.

\begin{enumerate}[label=\textup{(\alph*)},leftmargin=*]

\item {\bf Isotropic case.}
If $c=s=\pm1/\sqrt2$, then $\C=0$ and
\begin{align*}
\Sigma\cong \Hyp^3(-1/2).
\end{align*}

\item {\bf Product case.}
If $s=0$ and $c=\pm1$, then $\C=c=\pm1$ and
\begin{align*}
\Sigma\cong\Hyp^2(-1)\times\R.
\end{align*}
Analogously, if $c=0$ and $s=\pm1$, then $\C=-s=\mp1$ and
\begin{align*}
\Sigma\cong\R\times\Hyp^2(-1).
\end{align*}

\item {\bf Sol case.}
If $c=-s=\pm1/\sqrt2$, then $|\C|=\sqrt2$ and $\Sigma$ is a left-invariant 
metric of the $\mathrm{Sol}_3$ geometry.

\item {\bf Anisotropic case.}
In the remaining cases, that is, when
\begin{align*}
cs\neq0 \qquad\text{and}\qquad c\neq\pm s,
\end{align*}
the metric is a left-invariant metric on the semidirect product Lie group $G_{c,s}$.

\end{enumerate}
\end{theorem}

\begin{proof}
We start from the equilibrium metric
\begin{align*}
ds_\Sigma^2 = e^{-2ct}du^2+e^{-2st}dv^2+dt^2.
\end{align*}

If $c=s=a=\pm1/\sqrt2$, then
\begin{align*}
ds_\Sigma^2=e^{-2at}(du^2+dv^2)+dt^2.
\end{align*}
Taking $r=e^{at}$, we have $dt=dr/(ar)$ and, since $a^2=1/2$,
\begin{align*}
ds_\Sigma^2 = \frac{du^2+dv^2+2\,dr^2}{r^2}.
\end{align*}
With the change of variables $U=u/\sqrt2$, $V=v/\sqrt2$, we obtain
\begin{align*}
ds_\Sigma^2 = 2\,\frac{dU^2+dV^2+dr^2}{r^2}.
\end{align*}
Therefore, the metric is twice the standard hyperbolic metric in the upper 
half-space, and has constant sectional curvature $-1/2$. Thus
\begin{align*}
\Sigma\cong\Hyp^3(-1/2).
\end{align*}

If $s=0$ and $c=\pm1$, then
\begin{align*}
ds_\Sigma^2=e^{-2ct}du^2+dt^2+dv^2.
\end{align*}
Taking $r=e^{ct}$, we obtain
\begin{align*}
e^{-2ct}du^2+dt^2 = \frac{du^2+dr^2}{r^2},
\end{align*}
which is the metric of the hyperbolic half-plane of curvature $-1$. The 
$v$ direction is Euclidean, and therefore
\begin{align*}
\Sigma\cong\Hyp^2(-1)\times\R.
\end{align*}
The case $c=0$ and $s=\pm1$ is analogous, swapping the roles of $u$ and $v$.

If $c=-s=a=\pm1/\sqrt2$, then
\begin{align*}
ds_\Sigma^2 = e^{-2at}du^2+e^{2at}dv^2+dt^2.
\end{align*}
This is the standard diagonal form of a left-invariant metric of the 
$\mathrm{Sol}_3$ geometry, up to normalization of the vertical parameter.

Finally, in the remaining cases, the metric has already been identified as a 
left-invariant metric on the semidirect product Lie group $G_{c,s}$. The exclusions 
$cs\neq0$ and $c\neq\pm s$ precisely remove the hyperbolic, product, and 
Sol cases. Thus, the anisotropic case remains.
\end{proof}

\begin{table}[ht]
\centering
\renewcommand{\arraystretch}{1.65}
\setlength{\tabcolsep}{8pt}
\begin{tabular}{@{}cccll@{}}
\toprule
$\theta_\C$ & $H$ & $(c,s)$ & Regime & Geometry of $\Sigma$ \\
\midrule

$\dfrac{\pi}{4}$
&
$0$
&
$\left(\dfrac{1}{\sqrt2},\dfrac{1}{\sqrt2}\right)$
&
minimal
&
$\Hyp^3\!\left(-\dfrac12\right)$
\\
\addlinespace[4pt]

$\dfrac{5\pi}{4}$
&
$0$
&
$\left(-\dfrac{1}{\sqrt2},-\dfrac{1}{\sqrt2}\right)$
&
minimal
&
$\Hyp^3\!\left(-\dfrac12\right)$
\\

\midrule

$0$
&
$\dfrac{1}{3}$
&
$(1,0)$
&
subcritical
&
$\Hyp^2(-1)\times\R$
\\
\addlinespace[4pt]

$\pi$
&
$-\dfrac{1}{3}$
&
$(-1,0)$
&
subcritical
&
$\Hyp^2(-1)\times\R$
\\
\addlinespace[4pt]

$\dfrac{\pi}{2}$
&
$-\dfrac{1}{3}$
&
$(0,1)$
&
subcritical
&
$\Hyp^2(-1)\times\R$
\\
\addlinespace[4pt]

$\dfrac{3\pi}{2}$
&
$\dfrac{1}{3}$
&
$(0,-1)$
&
subcritical
&
$\Hyp^2(-1)\times\R$
\\

\midrule

$\dfrac{3\pi}{4}$
&
$-\dfrac{\sqrt2}{3}$
&
$\left(-\dfrac{1}{\sqrt2},\dfrac{1}{\sqrt2}\right)$
&
critical
&
$\mathrm{Sol}_3$
\\
\addlinespace[4pt]

$\dfrac{7\pi}{4}$
&
$\dfrac{\sqrt2}{3}$
&
$\left(\dfrac{1}{\sqrt2},-\dfrac{1}{\sqrt2}\right)$
&
critical
&
$\mathrm{Sol}_3$
\\

\midrule

otherwise
&
$\dfrac{c-s}{3}$
&
$(c,s)$
&
subcritical
&
$G_{c,s}$
\\

\bottomrule
\end{tabular}
\caption{Equilibrium solutions of the angular equation. Here
$c=\cos\theta_\C$, $s=\sin\theta_\C$, and
$H=(c-s)/3$ is the signed mean curvature with respect to the chosen normal.
The last row represents all remaining equilibrium values of $\theta_\C$.}
\label{tab:equilibria}
\end{table}

\end{document}